\documentclass[11pt,a4paper,leqno,twoside,notitlepage]{article}

\usepackage{amsmath}
\usepackage{amsfonts}
\usepackage{amssymb}
\usepackage{amsxtra}
\usepackage{amsthm}
\usepackage{enumerate}
\usepackage{ifthen}
\usepackage{xypic}

\newboolean{details}

\newenvironment{enumerateroman}
{\begin{enumerate}}
{\end{enumerate}}

\theoremstyle{plain}
\newtheorem{thm}[equation]{Theorem}
\newtheorem*{thm*}{Theorem}
\newtheorem{prop}[equation]{Proposition} 
\newtheorem*{prop*}{Proposition}
\newtheorem{lem}[equation]{Lemma}
\newtheorem{cor}[equation]{Corollary}

\theoremstyle{definition}
\newtheorem{defn}[equation]{Definition}
\newtheorem{rem}[equation]{Remark}

\newcommand{\mN}{{\mathbb N}}

\newcommand{\mcE}{\mathcal{E}}
\newcommand{\mcF}{\mathcal{F}}

\newcommand{\mcV}{\mathcal{V}}

\newcommand{\frakg}{\mathfrak{g}}

\renewcommand{\to}{\longrightarrow}

\newcommand{\mdot}{\!\cdot\!}

\providecommand{\im}{\mathop{\rm im}\nolimits}
\providecommand{\gr}{\mathop{\rm gr}\nolimits}

\providecommand{\ad}{\mathop{\rm ad}\nolimits}

\providecommand{\ext}{\mathop{\rm ext}\nolimits}
\providecommand{\Id}{\mathop{\rm Id}\nolimits}

\providecommand{\End}{\mathop{\rm End}\nolimits}

\providecommand{\Hom}{\mathop{\rm Hom}\nolimits}

\providecommand{\Ann}{\mathop{\rm Ann}\nolimits}

\providecommand{\Prim}{\mathop{\rm Prim}\nolimits}
\providecommand{\Spec}{\mathop{\rm Spec}\nolimits}

\providecommand{\res}{\mathop{\rm res}\nolimits}

\setboolean{details}{false}

\begin{document}
\centerline{\LARGE\bfseries{The commutant of simple modules over}}
\centerline{\LARGE\bfseries{almost commutative algebras}}
\vspace{1.5cm}
\centerline{\Large Oliver Ungermann}
\vspace{1cm}
\centerline{Ernst-Moritz-Arndt-Universit\"at Greifswald}
\centerline{Institut f\"ur Mathematik und Informatik}
\centerline{Friedrich-Ludwig-Jahn-Stra\ss e 15a}
\centerline{17487 Greifswald, Germany}
\vspace{1cm}
\centerline{\Large October 2009}
\vspace{1cm}

\begin{abstract}
Let $B$ be a finitely generated algebra over a field $k$. Then $B$ is called a
Jacobson algebra if every semiprime ideal of $B$ is semiprimitive. We will discuss
several conditions, all involving the commutant of simple $B$-modules, which imply
that $B$ is Jacobson. In particular, we will recover the well-known result
that every finitely generated almost commutative algebra is Jacobson. The same
holds true for $\mathbb{N}$-filtered $k$-algebras $B$ with a locally finite
filtration such that the associated graded $k$-algebra is left-noetherian.
\end{abstract}

\subsection*{Introduction}
\noindent Following Dixmier~\cite{Dix6} we shall recall the proofs
of the following three well-known results which are fundamental for
the representation theory of universal enveloping algebras: Let $k$
be a field and $\frakg$ a finite-dimensional Lie algebra over $k$.
Let $U(\frakg)$ denote the universal enveloping algebra of $\frakg$.
Then we have:
\begin{enumerate}
\item If $M$ is a simple $U(\frakg)$-module, then its commutant
$\End_{U(\frakg)}(M)$ is algebraic over~$k$. (This assertion is not
obvious if $M$ is infinite-dimensional as a $k$-vector space.)
\item The associative $k$-algebra $U(\frakg)$ is a Jacobson algebra
which means that every prime ideal is an intersection of primitive ideals.
\item In particular $U(\frakg)$ is Jacobson-semisimple.
\end{enumerate}
These assertions can be proved using variants of Grothendieck's generic
flatness lemma. In algebraic geometry generic flatness means the
following: If $A$ is a noetherian integral domain, $B$ a finitely
generated commutative $A-$algebra, and $M$ a finitely generated
$B$-module, then there exists a non-zero $f\in A$ such that the
localization $M_f=A_f\otimes_A M$ of~$M$ is free and hence flat.
Duflo~\cite{Duflo5} and Quillen~\cite{Quillen} showed that generic
freeness over $k[X]$ is satisfied for simple $k[X]\otimes_k B$-modules
where $B$ is a finitely generated almost commutative $k$-algebra,
compare Theorem~\ref{comm_thm:Q0_for_almost_commutative_algebras} and
Theorem~\ref{comm_thm:almost_commutative_algebras_are_Jacobson}. This
includes all quotients~$B$ of universal enveloping algebras of finite-dimensional
Lie algebras. Artin, Small, and Zhang~\cite{ArtinSmallZhang} proved that
generic freeness holds for simple $k[X]\otimes_k B$-modules where $B$ is
a $k$-algebra endowed with a locally finite $\mN$-filtration such that
the associated graded algebra is noetherian.\\\\
These results about generic freeness can be used to prove that algebras
over commutative Jacobson rings are also Jacobson. This was done by
Duflo~\cite{Duflo5} in the almost commutative case and by
Szczepanski~\cite{Szczepanski} in the $\mN$-filtered case.\\\\
As a motivation we give an application of the above results. Let $k$ be
an algebraically closed field. Let $P$ be a primitive ideal of $U(\frakg)$
and $M$ a simple $U(\frakg)$-module such that $P=\Ann_{U(\frakg)}(M)$,
keeping in mind that this equality does not determine $M$ up to equivalence.
In this case it follows $\End_{U(\frakg)}(M)=k\mdot 1$ because
$\End_{U(\frakg)}(M)$ is algebraic over $k$ and $k$ is algebraically
closed. Let $Z(\frakg)$ denote the center of $U(\frakg)$. Since
$Z(\frakg)/(Z(\frakg)\cap P)$ can be regarded as a subalgebra
of $\End_{U(\frakg)}(M)$, we conclude $Z(\frakg)/(Z(\frakg)\cap P)\cong k$.
Thus there exists a unique homomorphism
$\chi_P:Z(\frakg)\to k$ such that $\ker\chi_P=Z(\frakg)\cap P$,
which is called the central character of $P$. Since $U(\frakg)$ is
semiprimitive, we know that $W\in Z(\frakg)$ and $\chi_P(W)=0$ for all
$P\in\Prim(U(\frakg))$ implies $W=0$.\\\\
Let $S(\frakg)$ be the symmetric algebra of $\frakg$ and
$\bar{\gamma}:S(\frakg)\to U(\frakg)$ the Duflo map which is, in a certain
sense not to be explained here, a modification of the symmetrization map
\[\beta(X_1\ldots X_n)=\frac{1}{n!}\sum_{\sigma\in S_n} X_{\sigma(1)}
\ldots X_{\sigma(n)}.\]
The Poincar\'e-Birkhoff-Witt theorem implies that $\beta$ and $\bar{\gamma}$ are
$\ad(\frakg)$-invariant linear isomorphisms. Let $I(\frakg)$ denote the subalgebra
of all $\ad(\frakg)$-invariants of $S(\frakg)$. In 1969 Duflo proved that the
restriction $\gamma:I(\frakg)\to Z(\frakg)$ of $\bar{\gamma}$ is an
isomorphism of algebras, the so-called Duflo isomorphism. See~\cite{Duflo3}
for the case of solvable and semisimple Lie algebras, and ~\cite{Duflo4} for
the general case. Duflo's idea can be described as follows: In order to
prove $\gamma(X_1X_2)=\gamma(X_1)\gamma(X_2)$ for all $X_1,X_2\in I(\frakg)$,
it suffices, in view of the above considerations, to verify
\[\chi_P(\gamma(X_1X_2))=\chi_P(\gamma(X_1))\;\chi_P(\gamma(X_2))\]
for all $P\in\Prim(U(\frakg))$. The latter identity is an immediate consequence
of Duflo's character formula which states: If $f\in\frakg^\ast$, $P$ is the primitive
ideal of $U(\frakg)$ associated to $f$ via the Kirillov map, and $\chi_f=\chi_P$
is its central character, then
\[\chi_f(\gamma(X))=X(f)\]\
for all $X\in I(\frakg)$. But the proof of Duflo's character formula is demanding
and requires a detailed knowledge of the representation theory and the primitive
ideal theory of $U(\frakg)$.\\\\
The importance of the results $U(\frakg)$ semiprimitive and $\End_{U(\frakg)}(M)$
algebraic over $k$ in the context of Duflo isomorphism motivated us to elaborate
on their proofs.

\subsection*{Jacobson algebras}\label{sec:Jacobson_algebras}

Let $B$ be a unital ring. An ideal $P$ of $B$ is called primitive if there
exists a (non-zero) simple $B$-module $M$ with annihilator
\[P=\Ann_B(M)=\{a\in B:b\mdot m=0\text{ for all }m\in M\}\;.\]
An ideal $I$ of $B$ is said to be semiprimitive if it is an intersection of
primitive ones. The \textbf{Jacobson radical} $J(B)$ is defined as the
intersection of all primitive ideals of $B$.\\\\
We shall recall some facts about simple modules. If $L$ is a maximal left
ideal of $B$, then $M=B/L$ is a simple $B$-module whose annihilator is given
by
\[\Ann_B(M)=\{a\in B:a\mdot B\subset L\}\;.\]
Conversely, if $M$ is a simple $B$-module and $\xi\in M$ is non-zero, then
$L=\{a\in B:a\mdot\xi=0\}$ is a maximal left ideal and $M\cong B/L$ in a
canonical way. Since maximal left ideals exist by Zorn's Lemma, the
existence of simple modules and primitive ideals is guaranteed.\\\\
It is known that maximal ideals are primitive: Let $B$ be a ring and
$I$ a maximal ideal of $B$. By Zorn's Lemma there exists a maximal
left ideal $L$ of $B$ with $I\subset L$. Now $M=B/L$ is a simple
$B$-module with $I\subset\Ann_B(M)$. Since $I$ is maximal, we see
that $I=\Ann_B(M)$ is primitive.

\begin{lem}
If $B$ is commutative, then every primitive ideal of $B$ is maximal.
\end{lem}
\begin{proof}
Let $I$ be a primitive ideal of $B$, $I=\Ann_B(M)$ for some simple
$B$-module $M$. Let $\xi\in M$ be non-zero. Then $L=\{a\in B:a\mdot\xi=0\}$
is a maximal (left) ideal of~$B$. Clearly $I\subset L$. If $a\in I$, then
it follows $a\mdot M=a\mdot(B\mdot\xi)=B\mdot(a\mdot\xi)=0$. Thus
$I=L$ is maximal.
\end{proof}

\noindent We refrain from giving the various characterizations of the
Jacobson radical and content ourselves with the following observation.

\begin{lem}\label{comm_lem:1+a_invertible}
If $a\in J(B)$, then $1+xa$ is left invertible for all $x\in B$.
\end{lem}
\begin{proof}
Suppose that $1+xa$ is not left invertible. Then $B(1+xa)\neq B$. By
Zorn's Lemma there exists a maximal left ideal $L$ with $1+xa\in L$
and $1\not\in L$. Since $B/L$ is a simple $B$-module and $a\in J(B)$,
it follows $a\in aB\subset L$. This implies $1=(1+xa)-xa\in L$, a
contradiction.
\end{proof}

\noindent Next we shall introduce the Baer radical of $B$. Let $I$ be
an ideal of $B$ with $I\neq B$. We say that $I$ is prime if the following
condition is satisfied: If $J_1$ and $J_2$ are ideals of $B$ with
$J_1J_2\subset I$, then $J_1\subset I$ or $J_2\subset I$. The radical
$\sqrt{I}$ of $I$ is defined as the intersection of all prime ideals
containing $I$. Further we say that $I$ is semiprime if $I=\sqrt{I}$.
By Zorn's Lemma there exist maximal ideals containing $I$. And since
maximal ideals are prime, this implies that $\sqrt{I}\neq B$ is
well-defined. Clearly $I\subset\sqrt{I}$ and $\sqrt{I}=\sqrt{\sqrt{I}}$.
This shows that $\sqrt{I}$ is the smallest semiprime ideal of $B$
containing~$I$. Finally the \textbf{Baer radical} $\sqrt{0}$ is
defined to be the intersection of all prime ideals of $B$. In
the literature $\sqrt{0}$ is also known as the prime radical, the
lower nilradical or the Baer-McCoy radical of $B$.\\\\
It is easy to see that primitive ideals are prime: Let $M$ be a simple
module such that $I=\Ann_B(M)$. Let $J_1$, $J_2$ be ideals of $B$ such
that $J_1J_2\subset I$. If $J_2\not\subset I$, then $J_2\mdot M$ is a
non-zero submodule and hence $J_2\mdot M=M$ because $M$ is simple. This
yields $J_1\mdot M=J_1J_2\mdot M=I\mdot M=0$ and hence $J_1\subset I$.
Thus $I$ is prime. In particular we see that the Jacobson radical
contains the Baer radical: $\sqrt{0}\subset J(B)$.\\\\
In the proof of the next lemma we will need the following fact:
If $I$ is a semiprime ideal of $B$ and $J$ is an ideal of $B$ with
$J^2\subset I$, then it follows $J\subset I$. This can be seen as
follows: If $P$ is prime ideal with $J^2\subset I\subset P$, then
it follows $J\subset P$. Intersecting all these $P$ gives
$J\subset\sqrt{I}=I$.

\begin{lem}\label{comm_lem:Baer_radical_contains_nilpotent_ideals}
If $I$ is a nilpotent ideal of $B$, then $I\subset\sqrt{0}$.
\end{lem}
\begin{proof}
Suppose that $I$ is nilpotent and $I\not\subset\sqrt{0}$. Then
there is some $n\ge 1$ such that $I^n\not\subset\sqrt{0}$ and
$I^{n+1}\subset\sqrt{0}$. Let $k$ be the smallest integer
$\ge(n+1)/2$. Clearly $(I^k)^2\subset I^{n+1}\subset\sqrt{0}$.
Since $\sqrt{0}$ is semiprime, this implies $I^k\subset\sqrt{0}$,
a contradiction.
\end{proof}

\noindent A ring $B$ is called (semi-)prime or (semi-)primitive if
its zero ideal is (semi-)prime or (semi-)primitive respectively.
In the literature semiprimitive rings are also known as
Jacobson-semisimple rings. Clearly $I$ is a semiprime or
semiprimitive ideal of $B$ if and only if $B/I$ is a
semiprime or semiprimitive ring respectively.\\\\
For noetherian rings we obtain additional results.

\begin{prop}\label{comm_prop:Baer_radical_is_nilpotent}
Let $B$ be a left noetherian ring. Then the Baer radical $\sqrt{0}$
is a nilpotent ideal of $B$. Thus $B$ has a maximal nilpotent ideal.
\end{prop}
\begin{proof}
The first step is to prove that $B$ contains a maximal nilpotent
ideal. Suppose that this is not the case. Let $J_0$ be an arbitrary
nilpotent ideal of $B$. Since $J_0$ is not maximal, there exists
a nilpotent ideal $J_1$ of $B$ such that $J_1\not\subset J_0$. 
Clearly $J_0+J_1$ is a nilpotent ideal which properly contains
$J_0$. By induction we obtain a strictly increasing sequence of
nilpotent (two-sided) ideals, in contradiction to the assumption
that $B$ left noetherian.\\\\
Let $J$ denote the maximal nilpotent ideal of $B$.
Clearly $J\subset\sqrt{0}$ by Lemma~
\ref{comm_lem:Baer_radical_contains_nilpotent_ideals}.
Suppose that $J$ is not semiprime. Then Proposition~
\ref{comm_prop:characterization_of_semiprime_ideals}.\textit{(ii)}
of the Appendix shows that there exists an ideal $J_1$ of $B$
with $J_1^2\subset J$ and $J_1\not\subset J$. Since $J$ is
nilpotent, it follows that $J+J_1$ is also nilpotent, in
contradiction to the maximality of $J$. Thus $J$ is semiprime.
This shows that $\sqrt{0}\subset\sqrt{J}=J$ so that
$J=\sqrt{0}$ is nilpotent.
\end{proof}

\noindent The next proposition is crucial for the proof of
Proposition~\ref{comm_prop:Baer_radical_equals_Jacobson_radical}.

\begin{prop}\label{comm_prop:nil_ideals_are_nilpotent}
Let $B$ be a left noetherian ring. Then every nil right ideal of $B$
is nilpotent.
\end{prop}
\begin{proof}
Let $R$ be a nil right ideal of $B$. Let $J=\sqrt{0}$ denote the Baer
radical of $B$ which is known to be the maximal nilpotent ideal of $B$
by Proposition~\ref{comm_prop:Baer_radical_is_nilpotent}. Consequently
we must prove $R\subset J$. Suppose that $R\not\subset J$. Since $B$
is right noetherian, there exists some $a\in R\setminus J$ such that
the left ideal $l(a)=\{x\in B:xa\in J\}$ is maximal among all left
ideals of this form. Let $y\in B$ be arbitrary. Our aim ist to prove
$aya\in J$. If $ay\in J$, then we are done. So let us assume
$ay\not\in J$. For $ay\in L$ and $L$ is nil, there exists some
$k>1$ such that $(ay)^{k-1}\not\in J$ and $(ay)^k\in J$. Since
$l(a)\subset l(\,(ay)^{k-1}\,)$, it follows $l(\,(ay)^{k-1}\,)=l(a)$
by the maximality of $l(a)$. This implies $ay\in l(a)$ and hence
$aya\in J$. Since $aBa\subset J$ and $J$ is semiprime, it follows
$a\in J$. This contradiction proves $R\subset J$.
\end{proof}

\noindent The preceding two propositions as well as the subsequent
theorem remain valid after exchanging the words left and right.\\\\
The next result is fundamental.

\begin{lem}[\textbf{Schur}]\label{comm_lem:Schur}
If $M$ is a simple left $B$-module, then the commutant
\[D=\End_B(M)=\{\varphi:M\to M: \varphi \text{ is }B\text{-linear}\}\]
is a division ring. In particular the center $Z(D)$ of $D$ is a field.
\end{lem}
\begin{proof}
First of all, $D=\End_B(M)$ is a (not necessarily commutative) unital
ring. If $\varphi\in D$, $\varphi\neq 0$, then it follows $\im\varphi=M$
and $\ker\varphi=0$ because $\im\varphi\neq 0$ and $\ker\varphi\neq M$
are submodules of the simple $B$-module $M$. This means that $\varphi$
is bijective so that $\varphi^{-1}$ exists. Clearly $\varphi^{-1}\in D$.
This proves $D$ to be a division ring. Obviously $Z(D)$ is a field.
\end{proof}

In order to obtain information about the given ring $B$ we will
introduce an additional variable $X$ and consider (simple) modules $M$
over the ring $B[X]$ of all polynomials in one indeterminate with
coefficients in $B$. This striking idea can be traced back to the
famous one-sided article~\cite{Rabin} of Rabinovich from 1929.
Note that any $B[X]$-module structure boils down to a $B$-module $M$
together with a map $\varphi(m)=X\mdot m$ in its commutant. Clearly
$M$ is $B[X]$-simple if and only if $M$ does not admit proper
$\varphi$-invariant $B$-submodules.\\\\
Now let $A$ be a commutative unital ring and $B$ a unital $A$-algebra.
One might think of $A$ as a subring of the center $Z(B)$ of the
ring $B$. Note that the commutant $\End_B(M)$ of any $B$-module $M$
becomes an $A$-algebra via $(z\mdot\varphi)(m)=z\mdot\varphi(m)$.
We say that $\varphi\in\End_B(M)$ is integral over $A$ if there
exist elements $z_0,\ldots,z_{n-1}\in A$ such that
$\varphi^n+\sum_{k=0}^{n-1}z_k\varphi^k=0$. In the sequel we shall
investigate whether the following condition is satisfied:
\begin{itemize}
\item[\textbf{(E)}] If $M$ is a simple $B[X]$-module, then
$\varphi(m)=X\mdot m$ is integral over $A$.
\end{itemize}
Let $P=\Ann_{B[X]}(M)$ denote the primitive ideal of $B[X]$ associated
to $M$. Since $A\cap P=\Ann_A(M)$ is a prime ideal of $A$, it follows
that $A/A\cap P$ is an integral domain. Note that $M$ can be regarded as an
$A/A\cap P$-module. We point out that \textbf{(E)} is satisfied if
and only if $\varphi(m)=X\mdot m$ is integral over $A/A\cap P$. This is
the case if and only if $X$ is integral as an element of the integral
domain and $A/A\cap P$-algebra $A[X]/A[X]\cap P$.\\\\
The next proposition is due to Duflo, see Th\'eor\`eme~1
of~\cite{Duflo5}. This result seems to be inspired by an earlier
work of Amitsur on the radical of polynomial rings, see~\cite{Amitsur}.

\begin{prop}\label{comm_prop:Baer_radical_equals_Jacobson_radical}
Let $A$ be a commutative ring and $B$ be a left noetherian $A$-algebra.
If $B$ satisfies \textbf{(E)}, then the Jacobson radical and the Baer
radical of $B$ coincide.
\end{prop}
\begin{proof}
It suffices to show that $J(B)\subset\sqrt{0}$. Let $a\in J(B)$ be arbitrary.
First we shall prove that $B[X](1-aX)=B[X]$. Suppose that this is not the case
so that there exists a maximal left ideal $L$ of $B[X]$ such that $1-aX\in L$
and $1\not\in L$. Then $M=B[X]/L$ is a simple $B[X]$-module and $1+L\in M$
is non-zero. Since $X$ is in the center of $B[X]$, the map
\[\varphi:M\to M,\;\varphi(p+L)=X\mdot(p+L)=Xp+L,\]
is $B[X]$-linear. Clearly $\varphi(1+L)=X+L\neq 0$. As $M$ is simple,
Lemma~\ref{comm_lem:Schur} implies that $\varphi$ is invertible. By
assumption $\varphi\in\End_{B[X]}(M)$ is integral over $A$. Thus there
exist $z_0,\ldots,z_{n-1}$ in~$A$ such that
$\varphi^n+\sum_{k=0}^{n-1}z_k\;\varphi^k=0$. Dividing by $\varphi^n$ we
obtain $1+f(\varphi^{-1})=0$ where $f=\sum_{k=0}^{n-1}z_kY^{n-k}$ is a
polynomial with coefficients in $A$. From $aX+L=1+L$ we deduce
\[\varphi^k(a^k\mdot(1+L))=a^kX^k+L=(aX)^k+L=(aX+L)^k=1+L\]
for all $k\ge 0$ which shows that
\[(1+f(a))\mdot(1+L)=(1+f(\varphi^{-1}))(1+L)=0\;.\]
On the other hand, $f(a)\in J(B)$ because $f$ has constant term zero. Hence
$1+f(a)$ is left invertible in $B$ by Lemma~\ref{comm_lem:1+a_invertible}. This
contradiction proves $B[X](1-aX)=B[X]$.\\\\
Thus we know that there exists a polynomial $p=\sum_{k=0}^m a_k\,X^k$ in $B[X]$
such that
\[1=p(1-aX)=a_0\;+\;\sum_{k=1}^m(a_k-a_{k-1}a)\,X^k\;+\;a_ma\,X^{m+1}\;.\]
Comparing the coefficients we find $a_k=a^k$ for $0\le k\le m$ and
$a^{m+1}=a_ma=0$, which proves $a$ to be nilpotent. Thus $J(B)$ is
a nil ideal. As $B$ is left-noetherian, Proposition~
\ref{comm_prop:nil_ideals_are_nilpotent} shows that $J(B)$ is a nilpotent
ideal. Finally Lemma~\ref{comm_lem:Baer_radical_contains_nilpotent_ideals}
implies that $J(B)$ is contained in $\sqrt{0}$.
\end{proof}

Next we recall the definition of a Jacobson algebra which goes
back to Duflo~\cite{Duflo5}. The motivation is to generalize the
assertion of Hilbert's Nullstellensatz to the non-commutative
setting. To this end we recall the following version of the
Nullstellensatz: Let $k$ be an algebraically closed field and
$I$ a proper ideal of the commutative polynomial algebra
$B=k[X_1,\ldots,X_n]$. Let $\mcV(I)=\{\,\lambda\in k^n:
p(\lambda)=0\text{ for all }p\in I\,\}$ be the set of common
zeros of $I$ and $r(I)$ the ideal of all $p\in B$ such that
$p^n\in I$ for some $n\ge 1$. The Nullstellensatz asserts
\[r(I)=\{\;p\in B:p(\lambda)=0\text{ for all }
\lambda\in\mcV(I)\;\}\;.\]
This means that $r(I)$ is equal to the intersection of the
maximal ideals $\{\,I_\lambda:\lambda\in\mcV(I)\,\}$ of $B$
where $I_\lambda=\{\,p\in B:p(\lambda)=0\,\}$. In
particular it follows $r(I)=\sqrt{I}$, in accordance with
the usual notation. If $k$ is an arbitrary field, then one
might expect that $\sqrt{I}$ equals the intersection of all
maximal ideals containing $I$, whereas these maximal ideals
need not be of the form $I_\lambda$ for some $\lambda\in k$.
Finally one might ask for a generalization of the assertion
of the Nullstellensatz to non-commutative rings~$B$. The
decisive idea is to replace maximal ideals by primitive ones.
This step is encouraged by the fact that primitive ideals
play a very prominent role in representation theory. Taking
into account that $I=\sqrt{I}$ semiprime is necessary for
$I$ to be an intersection of primitive ideals we end up with

\begin{defn}\label{comm_defn:Jacobson_ring}
A ring $B$ is called a Jacobson ring if every semiprime ideal
of $B$ is semiprimitive.
\end{defn}

The next theorem is an immediate consequence of
Proposition~\ref{comm_prop:Baer_radical_equals_Jacobson_radical}.

\begin{thm}\label{comm_thm:E_implies_Jacobson_property}
Let $A$ be a commutative ring and $B$ a left noetherian $A$-algebra
satisfying condition~\textbf{(E)}. Then $B$ is a Jacobson ring.
\end{thm}
\begin{proof}
Let $I$ be a semiprime ideal of $B$. Then $\bar{B}=B/I$ is a
left notherian semiprime ring. Note that $\bar{B}[X]$ is a
quotient of $B[X]$ so that every $\bar{B}[X]$-module can be
regarded as a $B[X]$-module. Thus condition~\textbf{(E)}
implies that $\varphi=X\mdot m$ is integral over $A$ for
every simple $\bar{B}[X]$-module $M$. Now Proposition~
\ref{comm_prop:Baer_radical_equals_Jacobson_radical}
yields $J(\bar{B})=\sqrt{\bar{B}}=0$ which shows that
$I$ is semiprimitive.
\end{proof}

Duflo proved that finitely generated almost commutative
algebras over commutative Jacobson rings are Jacobson,
see Th\'eor\`eme~3 of~\cite{Duflo5}. As a first step in this
direction he considers almost commutative algebras over
arbitrary fields. In this context the following condition
plays a crucial role: Let $k$ be a field and $B$ an arbitrary
$k$-algebra.

\begin{itemize}
\item[\textbf{(D)}] If $M$ is a simple $B[X]$-module, then
$\varphi(m)=X\mdot m$ is algebraic over $k$.
\end{itemize}
According to Theorem~\ref{comm_thm:E_implies_Jacobson_property}
this condition implies the Jacobson property. By contrast Irving
prefered the following stronger requirement.
\begin{itemize}
\item[\textbf{(I)}] If $M$ is a simple $B[X]$-module, then
$\End_{B[X]}(M)$ is algebraic over $k$.
\end{itemize}
However, we do not follow Irving~\cite{Irving} in saying that
$B[X]$ satisfies the Nullstellensatz if condition~\textbf{(I)}
holds true.\\\\
In~\cite{Quillen} Quillen discovered that the generic flatness lemma
can be used to prove that the commutant of a simple module is algebraic.
We shall explain this important fact in detail.\\\\
First we recall some definitions. Let $A$ be a commutative ring
and $f\in A$ not nilpotent defining the multiplicative subset
$S_f=\{f^k:k\ge 0\}$ of $A$. Then $A_f=S_f^{-1}A$ is called a simple
localization. Roughly speaking, $A_f$ is generated by $A$ and
$f^{-1}$. Note that $A_f\neq 0$ because $f$ is not nilpotent.
We say that an $A$-module $M$ is \textbf{generically free} if
there exists a non-nilpotent $f\in A$ such that $M_f=A_f\otimes_A M$
is a free $A_f$-module. In the sequel we will use the fact that the
localization functor is exact which means that $A_f$ is a flat
$A$-module.\\\\
Returning to the original situation we suppose that $B$ is a
$k$-algebra and that $M$ is a simple $B[X]$-module. Since $k[X]$
is contained in the center of $B[X]$, $M$ becomes a module over
the principal ideal domain $A=k[X]$ such that the actions of $A$
and $B[X]$ commute. Hence the set
\[T_A(M)=\{m\in M: a\mdot m=0\text{ for some }a\in A\}\]
of $A$-torsion elements is a $B[X]$-submodule of $M$. If $T_A(M)\neq 0$,
then there exist non-zero elements $m\in M$ and $f\in A$ such that
$f\mdot m=0$. Then
\[f\mdot M=f\mdot(B[X]\mdot m)=B[X]\mdot(f\mdot m)=0\]
which means $f(\varphi)=0$. Thus $\varphi(m)=X\mdot m$ is algebraic
over $k$. In order to establish property~\textbf{(D)} we must exclude
the case $T_A(M)=0$. (In this case $M$ is flat, but not necessarily
free over $A$.)\\\\
According to Quillen we introduce the following condition.

\begin{itemize}
\item[\textbf{(Q0)}] Every simple $B[X]$-module is generically
free over $k[X]$.
\end{itemize}

\noindent In the proof of Theorem~\ref{comm_thm:Q0_implies_D} we will
need the following auxiliary result.

\begin{lem}\label{comm_lem:finitely_many_irreducibles}
Let $k$ be a field. Then $k[X]$ contains infinitely many monic irreducibles.
\end{lem}
\begin{proof}
If $k$ is infinite, then $J=\{X-\lambda:\lambda\in k\}$ is an infinite
set of irreducibles with leading coefficient $1$. Next we assume that $k$
is finite. Suppose that the set $J$ of all irreducible and monic polynomials
in $k[X]$ is finite. Set $n=\max\{\deg(q):q\in J\}$ and define $Q$ as
the product of all $q\in J$. Since $\deg(Q)>n$, it follows that $1+Q$
is reducible. Hence there exist $p,q\in k[X]$ with $1+Q=pq$ and $p$
irreducible. Clearly $p\,|\,Q$ so that there is a $q'\in k[X]$ with
$Q=pq'$. This implies $pq=1+Q=1+pq'$ and $p(q-q')=1$, a contradiction.
Thus $J$ is infinite.
\end{proof}

\noindent The next result is due to Quillen, see~\cite{Quillen}. Since the proof
in~\cite{Quillen} is quite succinct, we reproduce the elaborate argument given
by Dixmier in Lemme~2.6.4 of~\cite{Dix6}.

\begin{thm}\label{comm_thm:Q0_implies_D}
Let $B$ be a unital $k$-algebra. Then \textbf{(Q0)} implies \textbf{(D)}.
\end{thm}
\begin{proof}
Let $M$ be a simple $B[X]$-module. As above we regard $M$ as a module
over $A=k[X]$. Suppose that $\varphi(m)=X\mdot m$ is transcendental over $k$.
By \textbf{(Q0)} there exists a non-zero $f\in A$ such that $M_f=A_f\otimes_A M$
is a free $A_f$-module.\\\\
First of all we observe that the natural map $\eta:M\to M_f$,
$\eta(m)=1\otimes m$, is an injection: Its kernel $\ker\eta=T_{S_f}(M)=
\{m\in M: f^n\mdot m=0\text{ for some }n\ge0\}$ is trivial because
$\varphi$ is transcendental. Thus $M_f\neq 0$.\\\\
By Lemma~\ref{comm_lem:finitely_many_irreducibles} we can choose an
irreducible polynomial $g\in A$ not dividing~$f$. Then $A_f\to A_f$,
$p\mapsto gp$ is not surjective because $f^n\not\in gA$ for all $n\ge 0$.\\\\
Next we define $\Psi:M_f\to M_f,\; \Psi(p\otimes m)=(gp)\otimes m=
p\otimes(g\mdot m)$. On the one hand, $M_f\cong A_f^{(I)}\neq 0$ is
free and $\Psi:A_f^{(I)}\to A_f^{(I)}$ has the form $\Psi(p)_i=gp_i$.
This shows that $\Psi$ is not surjective. On the other hand, $\im g(\varphi)$
is a non-zero $B[X]$-submodule of $M$ because $\varphi\in\End_{B[X]}(M)$
is transcendental over $k$. As $M$ is simple, it follows $\im g(\varphi)=M$.
This shows $\im\Psi=M_f$ because $\Psi(p\otimes m)=p\otimes g(\varphi)(m)$ and
$M_f=A_f\otimes_A M$ is generated by elementary tensors. This contradiction
proves \textbf{(D)}. 
\end{proof}

Next we will prove that \textbf{(Q0)} holds for all finitely
generated almost commutative algebras. Let $A$ be a commutative
ring and $B$ a unital $A$-algebra. Suppose that there exists an
$A$-submodule $\frakg$ of $B$ such that $xy-yx\in\frakg$ for all
$x,y\in\frakg$ and $B=\sum_{n=0}^\infty \frakg^n$. In particular
this means that $\frakg$ is a Lie subalgebra of $B$ with respect
to the canonical Lie algebra structure $[x,y]=xy-yx$ of $B$.
Further we assume that $\frakg$ is a finitely generated $A$-module.
If $x_1,\ldots,x_d$ are generators of $\frakg$, then $\frakg^n$ is,
modulo $\frakg^{n-1}$ and as an $A$-module, generated by the set
of all products of the form $x^\nu=x_1^{\nu_1}\ldots x_d^{\nu_d}$
with $\nu\in\mN^d$ and $|\nu|=\sum_{j=0}^d \nu_j=n$. We say that
$(B,\frakg)$ is a finitely generated almost commutative $A$-algebra
if the above conditions are satisfied.\\\\
By means of the universal property of universal enveloping algebras
we obtain the following characterization.

\begin{lem}
Let $A$ be a commutative ring. If $\frakg$ is a Lie algebra over~$A$
which is finitely generated as an $A$-module, then every quotient
of its universal enveloping algebra $U(\frakg)$ is a finitely
generated almost commutative $A$-algebra. Conversely, if
$(B,\frakg)$ is finitely generated and almost commutative,
then $B$ is isomorphic to a quotient of $U(\frakg)$.
\end{lem}

Theorem~\ref{comm_thm:Q0_for_almost_commutative_algebras}
generalizes the generic flatness lemma of algebraic geometry, compare
Lemma~6.7 in~\cite{Groth1} Expos\'e IV of SGA~1960-61. Grothendieck's
proof of generic freeness for modules over finitely generated
commutative algebras uses the Noether normalization Lemma and
Krull dimension theory. Here we restate a more elementary proof
due to Dixmier~\cite{Dix7} and Duflo~\cite{Duflo5} only relying on
the fact that the localization functor is exact.\\\\
A filtration of an $A$-module $M$ is a sequence $M_n$ of
$A$-submodules of $M$ such that $M_{-1}=\{0\}$,
$M_{n-1}\subset M_n$ for all $n\ge 0$, and
$M=\sum_{n=0}^\infty M_n$.

\begin{lem}\label{comm_lem:main_trick}
Let $A$ be a commutative ring and $M$ a left $A$-module endowed
with a filtration $\{M_n:n\ge -1\}$. Suppose that $f\in A$ is
non-nilpotent such that $f\mdot M_n\subset M_{n-1}$ whenever
$M_n/M_{n-1}$ is not free. Then $A_f\otimes_A M$ is a free
$A_f$-module.
\end{lem}
\begin{proof}
The images $L_n$ of the canonical injections
$\iota_n:A_f\otimes_A M_n\to A_f\otimes_A M$ form a filtration
of the $A_f$-module $A_f\otimes_A M$. We claim that $L_n/L_{n-1}$
is free over $A_f$ for all $n\ge 0$: First we treat the case
$M_n/M_{n-1}$ not free. Then
\[r\otimes m=(rf^{-1})\otimes (f\mdot m)\in L_{n-1}\]
for all $r\in A_f$ and $m\in M_n$. This proves $L_n=L_{n-1}$ so that
$L_n/L_{n-1}=0$. Next we assume that $M_n/M_{n-1}\cong A^{(I)}$ is
free. As $A_f\otimes_A -$ commutes with direct sums, it follows that
$A_f\otimes_A(M_n/M_{n-1})\cong A_f\otimes_A A^{(I)}\cong A_f^{(I)}$
is a free $A_f$-module. Since $A_f$ is flat, it follows that
$L_n/L_{n-1}\cong A_f\otimes_A(M_n/M_{n-1})$ is free over $A_f$
for all $n\ge 0$. This shows that $A_f\otimes_A M$ is free.
\end{proof}

The next theorem is due to Duflo, see Th\'eor\`eme 2 of~\cite{Duflo5}.
The idea of the proof goes back to Dixmier~\cite{Dix7}. Compare also
Lemme~2.6.3 of Dixmier's book~\cite{Dix6}. The proof is a bit
technical and needs some preparation. For $d\ge 1$ we interpret
$\mN^d$ as the set of all multi-indices. The order of $\nu\in\mN^d$
is given by $|\nu|=\sum_{j=1}^d\nu_j$. Let $\le$ denote the unique
total order on $\mN^d$ satisfying the following two properties:
\begin{enumerate}
\item If $\mu,\nu\in\mN^d$ and $|\mu|<|\nu|$, then $\mu\le\nu$.
\item If $\mu,\nu\in\mN^d$ such that $|\mu|=|\nu|$ and if there
exists $1\le k\le n$ such that $\mu_j=\nu_j$ for $1\le j\le k-1$
and $\mu_k<\nu_k$, then $\mu\le\nu$.
\end{enumerate}

\noindent As a poset $(\mN^d,\le)$ is isomorphic to $(\mN,\le)$
in a canonical way. By abuse of notation we write $\nu-1$ for
the maximum of the set of all $\mu\in\mN^d$ such that $\mu\le\nu$
and $\mu\neq\nu$, provided that $\nu\neq 0$.

\begin{thm}\label{comm_thm:Q0_for_almost_commutative_algebras}
Let $A$ be an integral domain and $(B,\frakg)$ a finitely generated
almost commutative $A$-algebra. Then the following enhancement
of~\textbf{(Q0)} holds true: If $M$ a cyclic $B$-module, then
there is a non-zero $f\in A$ such that $A_f\otimes M$ is a
free $A_f$-module.
\end{thm}
\begin{proof}
Let $M$ be a cyclic $B$-module and $\xi\in M$ with $M=B\mdot\xi$. Let
$x_1,\ldots,x_d$ be generators of the $A$-submodule $\frakg$ of $B$. As
usual we write $x^\nu=x_1^{\nu_1}\ldots x_d^{\nu_d}$ for $\nu\in\mN^d$.
Since $B=\sum_{n=0}^\infty \frakg^n=\sum_{\nu\in\mN^d}A\mdot x^\nu$,
it is evident that $M_\nu=\sum_{\mu\le\nu}A\mdot x^\mu\mdot\xi$ 
defines a filtration of the $A$-module $M$ such that $M_\nu/M_{\nu-1}$ is
cyclic for all $\nu$. Let us consider the ideals
\[I_\nu=\{a\in A:a\mdot M_\nu\subset M_{\nu-1}\}=\{a\in A:a\mdot x^\nu
\mdot\xi\in M_{\nu-1}\}\]
of $A$. Clearly $M_\nu/M_{\nu-1}\cong A$ is free provided that
$I_\nu=0$. Let $\Lambda=\{\nu\in\mN^d:I_\nu\neq 0\}$. In view of
Lemma~\ref{comm_lem:main_trick} it now suffices to prove that
the ideal $J=\bigcap\,\{I_\nu:\nu\in\Lambda\}$ is non-zero. The
problem is to treat the case where $\Lambda$ happens to be infinite.\\\\
If $\nu\in\mN^d$ is arbitrary, $e_j\in\mN^d$ is the $j^{th}$ canonical
basis vector, and $a\in I_\nu$, then
\[a\mdot x^{\nu+e_j}\mdot\,\xi=x_j\mdot a\mdot x^\nu\mdot\,\xi+
a\mdot\,[x_1^{\nu_1}\ldots x_{j-1}^{\nu_{j-1}},x_j]\,x_j^{\nu_j}\ldots
x_d^{\nu_d}\,\mdot\,\xi\;\in M_{\nu+e_j-1}\]
which shows $a\in I_{\nu+e_j}$. Thus $I_\nu\subset I_{\nu+\alpha}$ for
all $\alpha\in\mN^d$ and hence $\Lambda+\mN^d=\Lambda$. Let $\Lambda_0$
denote the set of all $\gamma\in\Lambda$ such that $\gamma\neq\nu+e_j$
for all $\nu\in\Lambda$ and $1\le j\le n$. It is easy to see that
$\Lambda_0$ is a finite subset of $\Lambda $ such that
$\Lambda=\Lambda_0+\mN^d$.
Since $A$ is an integral domain and $I_\nu\subset I_{\nu+\alpha}$,
it follows that $J=\bigcap\{I_\gamma:\gamma\in\Lambda_0\}$ is non-zero.
For every $0\neq f\in J$ we have $f\mdot M_\nu\subset M_{\nu-1}$ for
all $\nu\in\Lambda$. This completes the proof.
\end{proof}

The preceding theorem can be generalized to the case of finitely
generated $B$-modules. We omit the details.\\\\
The next result appears as a corollary of the preceding achievements.
It is a generalization of a result of Krull in the commutative case,
see Satz~1 and Satz~2 of~\cite{Krull}. Compare also Theorem~3 of
Goldmann~\cite{Goldman}.

\begin{thm}\label{comm_thm:almost_commutative_algebras_are_Jacobson}
Let $k$ be a field and $(B,\frakg)$ a finitely generated almost
commutative $k$-algebra. Then $B$ satisfies condition~\textbf{(D)}.
In particular $B$ is a Jacobson algebra and the commutant
$\End_B(M)$ of any simple $B$-module is algebraic over $k$.
\end{thm}
\begin{proof}
By Theorem~\ref{comm_thm:Q0_for_almost_commutative_algebras}
we know that $B$ has property~\textbf{(Q0)}. Theorem~
\ref{comm_thm:Q0_implies_D} implies~\textbf{(D)}. Thus $B$ is
Jacobson by Theorem~\ref{comm_thm:E_implies_Jacobson_property}.
Since any $\varphi\in\End_B(M)$ gives rise to a simple
$B[X]$-module, it follows by~\textbf{(D)} that $\varphi$
is algebraic over $k$.
\end{proof}

\begin{lem}
Let $k$ be an algebraically closed field and $B$ be a finitely
generated commutative $k$-algebra. If $I$ is a primitive
(i.e.\ maximal) ideal of $B$, then there exists a homomorphism
$\chi:B\to k$ such that $I=\ker\chi$.
\end{lem}
\begin{proof}
Let $M$ be a simple $B$-module with $I=\Ann_B(M)$. Observe
that $B/I$ can be regarded as a subring of $\End_B(M)$
because $B$ is commutative. Since $k$ is algebraically
closed and $\End_B(M)$ is algebraic over $k$ by Theorem~
\ref{comm_thm:almost_commutative_algebras_are_Jacobson},
it follows $B/I\cong k$. Thus the canonical projection
$\chi:B\to B/I\cong k$ has the desired properties.
\end{proof}

Finally, applying these results to the commutative algebra
$B=k[X_1,\ldots,X_n]$, where $k$ algebraically closed, we
rediscover Hilbert's Nullstellensatz: Any semiprime ideal
$I=\sqrt{I}$ of $B$ is equal to the intersection of the
maximal ideals containing it. As the homomorphisms
$\chi:B\to k$ are known in this case, every maximal
ideal is of the form $I_\lambda=\{\;p\in B:p(\lambda)=0\;\}$
for some $\lambda\in k^n$.

\begin{thm}
Let $A$ be a commutative noetherian Jacobson ring and $B$
a finitely generated almost commutative $A$-algebra. Then
$B$ is a Jacobson algebra and the commutant $\End_B(M)$
of any simple $B$-module is integral over $A$.
\end{thm}
\begin{proof}
First of all, $B$ is known to be left noetherian. According
to Theorem~\ref{comm_thm:E_implies_Jacobson_property} we must
verify~\textbf{(E)}. Let $M$ be a simple $B[X]$-module with
annihilator $P=\Ann_B(M)$ and $\varphi(m)=X\mdot m$. As we
have already seen, it suffices to prove that $\varphi$ is
integral over the integral domain $\bar{A}=A/A\cap P$. For
the convenience of the reader we reproduce the proof of
Proposition~1 of~\cite{Duflo5} which states that $\bar{A}$
is a field.\\\\
By Theorem~\ref{comm_thm:Q0_for_almost_commutative_algebras}
there exists a non-zero $f\in\bar{A}$ such that
$M_f=\bar{A}_f\otimes_{\bar{A}}M$ is a free $\bar{A}_f$-module.
Suppose that the natural map $\eta:M\to M_f$, $\eta(m)=1\otimes m$,
has a non-trivial kernel $\ker\eta=T_{S_f}(M)\neq 0$. Then
there exists a $k\ge 1$ and a non-zero $m\in M$ such that
$f^k\mdot m=0$. Since $M$ is simple, it follows
$f^k\mdot M=0$. Hence $f^k=0$. For $\bar{A}$ has no zero
divisors, we get $f=0$, a contradiction. Thus $\eta$ is
injective. In particular $M_f$ is non-zero.\\\\
First we prove that $\bar{A}_f$ is a field. Let $a\in\bar{A}_f$
be non-zero. Clearly $\mu_a:M_f\to M_f$, $\mu_a(m)=a\mdot m$,
is a non-zero element of $\End_B(M)$. By Schur's Lemma $\mu_a$
is invertible. Let $\{m_i:i\in I\}$ be a basis of $M_f$ and
choose $i\in I$. One can find $b,b_j\in\bar{A}_f$ such that
\[\mu_a^{-1}(m_i)=b\mdot m_i+\sum_{j\neq i}b_j\mdot m_j\;.\]
Applying $\mu_a$ to both sides and comparing coefficients,
we obtain $ab=1$ so that $a$ is invertible in $\bar{A}_f$.
Consequently $\bar{A}_f$ is a field.\\\\
Since $\bar{A}$ is a Jacobson ring, there exists a maximal ideal
$I$ of $\bar{A}$ such that $f\not\in I$. As $f+I$ is invertible
in $\bar{A}/I$, there exists a unique ring homomorphism
$\psi:\bar{A}_f\to \bar{A}/I$ such that $\psi(x)=x+I$ for
all $x\in\bar{A}$. It follows that $\psi$ is injective
because $\bar{A}_f$ is a field. This proves $I=0$. Thus
$\bar{A}$ itself is a field.\\\\
Since $\bar{B}=B/P$ is an almost commutative $\bar{A}$-algebra,
Theorem~\ref{comm_thm:Q0_for_almost_commutative_algebras}
implies that $\varphi$ is algebraic over $\bar{A}$.
This proves~\textbf{(E)}.
\end{proof}

In respect of Irving's condition \textbf{(I)} we state the following
generic freeness property.
\begin{itemize}
\item[\textbf{(Q1)}]If $M$ is a simple $B[X]$-module and
$\varphi\in\End_{B[X]}(M)$, then $Y\mdot m=\varphi(m)$ defines a
generically free $k[Y]$-module.
\end{itemize}

In the situation of \textbf{(Q1)} we can regard $M$ as a simple
module over the $k$-algebra $B[X,Y]=B[X][Y]=B[X]\otimes_k k[Y]$.
Applying Theorem~\ref{comm_thm:Q0_implies_D} to $B[X][Y]$,
we obtain 

\begin{lem}\label{comm_cor:Q1_implies_I}
Condition \textbf{(Q1)} implies \textbf{(I)}.
\end{lem}

\noindent Finally we add the conditions \textbf{(A)} the commutant
$\End_{B}(M)$ of any simple $B$-module $M$ is algebraic over $k$ and
\textbf{(J)} $B$ is a Jacobson algebra. Altogether these conditions
are related as follows:\\\\
\makebox[15cm]{
\xymatrix{
\textbf{(Q1)} \ar@{=>}[d] \ar@{=>}[r] & \textbf{(Q0)} \ar@{=>}[d] \\
\textbf{(I)} \ar@{=>}[r] & \textbf{(D)} \ar@{=>}[r]
\ar@{=>}[d] & \textbf{(J)}\\
& \textbf{(A)}
}
}

\noindent Here the implication \textbf{(D)}$\Rightarrow$\textbf{(J)}
is valid only if $B$ is left (or right) noetherian. The others hold
true in general.

\subsection*{Modules over filtered algebras}

Some results of the preceding section can be generalized to
filtered $k$-algebras such that the associated graded algebra
is noetherian. To begin with, let $A$ be a commutative
ring and $B$ an $A$-algebra. An increasing sequence
$\mcF=\{B_n:n\ge -1\}$ of $A$-submodules of $B$ is
called a filtration of $B$ if the following conditions
are satisfied: $B_{-1}=0$, $A\,\mdot\,1\subset B_0$,
$\sum_{n=0}^\infty B_n=B$, and $B_mB_n\subset B_{m+n}$.
We say that $(B,\mcF)$ is an $\mN$-filtered $A$-algebra.
The filtration is locally finite if $B_n$ is a finitely
generated $A$-module for all $n\ge 0$.\\\\
Subject to a given filtration $\mcF$ of $B$ we define the
$A$-modules $\bar{B}_n=B_n/B_{n-1}$ and $\bar{B}=
\bigoplus_{n=0}^\infty \bar{B}_n$. It is easy to see that
\[\bar{B}_m\times \bar{B}_n\to \bar{B}_{m+n},\;
(x+B_{m-1})(y+B_{n-1})=xy+B_{m+n-1}\]
is well-defined and turns $\bar{B}$ into an $\mN$-graded
$A$-algebra. We call $\bar{B}=\gr(B,\mcF)$ the associated
graded algebra of $B$.\\\\
Let $B$ be an $A$-algebra. If $\frakg$ is an $A$-submodule
of $B$ such that $[\frakg,\frakg]\subset\frakg$ and
$B=\sum_{n=0}^\infty\frakg^n$, then $B_n=\frakg^n+B_{n-1}$
yields a filtration $\mcF$ of $B$ such that the associated
graded algebra $\gr(B,\mcF)$ is commutative, i.e., such that
$[B_n,B_m]\subset B_{n+m-1}$ for all $m,n\ge 0$. Generalizing
the notion of finitely generated almost commutative algebras
given in the previous section we state
\begin{defn}
A filtered $A$-algebra $(B,\mcF)$ is called almost commutative
if $\gr(B,\mcF)$ is commutative. If in addition $\gr(B,\mcF)$ is
finitely generated, then $(B,\mcF)$ is said to be of finite type.
\end{defn}

Conversely, if $(B,\mcF)$ is almost commutative, then $B_1$
is a Lie subalgebra of $B$, but $B=\sum_{n=0}^\infty B_1^n$ is
not apparent. If in addition $(B,\mcF)$ is of finite type, then
$B_1$ is finitely generated.\\\\
Suppose that $A$ is a noetherian commutative ring and $B$ is
an almost commutative $A$-algebra of finite type. Then it
follows by the Hilbert basis theorem that $\gr(B,\mcF)$ is
a noetherian $A$-algebra.\\\\
Let $(B,\mcF)$ be a filtered $A$-algebra. For every finitely
generated $B$-module $M$ with generators $\xi_1,\ldots,\xi_r$
we consider the filtration
\[M_n=\sum_{j=1}^r B_n\mdot \xi_j\]
of $M$. This means that $\mcE=\{M_n:n\ge -1\}$ is an increasing
sequence of $A$-submodules of $M$ such that $M_{-1}=0$,
$\sum_{n=1}^\infty M_n=M$, and $B_m\mdot M_n\subset M_{m+n}$.
Here the action of $A$ on $M$ is given by
$a\,\mdot\,\xi=(a\,\mdot\,1)\mdot\,x$. Note that the
definition of $\mcE$ depends on the choice of the
generators $\xi_1,\ldots,\xi_r$ of $M$ and the filtration
of $\mcF$ of $B$. Now we set $\bar{M}_n=M_n/M_{n-1}$ and
$\bar{M}=\bigoplus_{n=0}^\infty \bar{M}_n$.
One verifies easily that
\[\bar{B}_m\times \bar{M}_n\to\bar{M}_{m+n},\;(x+B_{m-1})\mdot
(m+L_{\nu-1})=x\mdot m+L_{m+n-1}\]
turns $\bar{M}=\gr(M,\mcE)$ into a finitely generated graded
$\bar{B}$-module.\\\\
As from now let $k$ be a field and $B$ a $k$-algebra with
filtration $\mcF=\{B_n:n\ge -1\}$. The $k[X]$-algebra $B'=B[X]$
carries a natural filtration $\mcF'$ given by
\[B'_n=B_n[X]=\{p\in B[X]:p(k)\in B_n\text{ for all }k\ge 0\}\]
and the associated graded $k[X]$-algebra $\gr(B',\mcF')$ is
isomorphic to $\bar{B}[X]=\gr(B,\mcF)[X]$.\\\\
We remind the reader that we are interested in finding convenient
conditions which are sufficient for $B$ to be a Jacobson
algebra. To this end we introduce

\begin{itemize}
\item[\textbf{(G0)}] Every cyclic graded $\bar{B}[X]$-module is
generically free over $k[X]$.
\end{itemize}

The next proposition can be found (more or less explicitly)
in the work of Quillen~\cite{Quillen}, Dixmier~\cite{Dix6},
and Artin-Small-Zhang~\cite{ArtinSmallZhang}.

\begin{prop}\label{comm_prop:G0_implies_Q0}
Let $(B,\mcF)$ be an $\mN$-filtered $k$-algebra. Then \textbf{(G0)}
implies \textbf{(Q0)}.
\end{prop}
\begin{proof}
Let $M$ be a simple $B[X]$-module and $\xi\in M$ non-zero. As above
we define a filtration $M_n=B_n[X]\mdot\xi$ of $M$ and consider the
associated graded $\bar{B}[X]$-module $\bar{M}$. Note that $\bar{M}$
is cyclic over $\bar{B}[X]$. Put $A=k[X]$. By \textbf{(G0)} there
exists a non-zero $f\in A$ such that
\[\bar{M}_f=A_f\otimes_A\bar{M}\;\cong\;\bigoplus_{n=0}^\infty\;
A_f\otimes_A (M_n/M_{n-1})\]
is free. (Here we used the fact that $A_f\otimes_A -$ commutes with
direct sums.) Since $A_f$ is a principal ideal domain, it follows
that the $A_f$-submodules $A_f\otimes_A(M_n/M_{n-1})$ of $\bar{M}_f$
are free for all $n$. We know that 
\begin{equation*}
0\longrightarrow A_f\otimes_A M_{n-1}\overset{\eta_{n-1}}{\longrightarrow}
A_f\otimes_A M_n \longrightarrow A_f\otimes_A (M_n/M_{n-1})
\longrightarrow 0
\end{equation*}
is exact because $A_f$ is flat over $A$. Consequently the canonical
images $L_n$ of $A_f\otimes_A M_n$ in $M_f=A_f\otimes_A M$ form a
filtration of $M_f$ such that
\[L_n/L_{n-1}\cong (A_f\otimes_A M_n)\;/\;\eta_{n-1}
(A_f\otimes_A M_{n-1})\cong A_f\otimes_A(M_n/M_{n-1})\]
is free over $A_f$ for all $n\ge 0$. Hence it follows that
$A_f\otimes_A M_f$ is a free $A_f$-module. This proves~\textbf{(Q0)}.
\end{proof}

In the sequel we shall restrict ourselves to filtered algebras
$(B,\mcF)$ such that $\gr(B,\mcF)$ is left noetherian. As we will see
next, a necessary condition for the latter is that $B$ itself is
left noetherian: Let $A$ be a ring and $(B,\mcF)$ an $A$-algebra with
filtration $\mcF=\{B_n:n\ge -1\}$. Let $\gr_n:B_n\to B_n/B_{n-1}$
denote the canonical maps. If $L$ is a left ideal of $B$, then
\[\gr(L)=\bigoplus_{n=0}^\infty\gr_n(L\cap B_n)\]
is a left ideal of $\gr(B)$. Clearly $L_1\subset L_2$ implies
$\gr(L_1)\subset\gr(L_2)$.

\begin{lem}\label{comm_lem:gr(L_1)_strictly_contained_in_gr(L_2)}
If $L_1$ and $L_2$ are left ideals of $B$ such that $L_1\subset L_2$
and $L_1\neq L_2$, then $\gr(L_1)\neq\gr(L_2)$.
\end{lem}
\begin{proof}
Suppose that $\gr(L_1)=\gr(L_2)$. Let $n\ge 0$ be minimal with
$L_1\cap B_n\neq L_2\cap B_n$. Choose $b\in L_2\cap B_n$ such
that $b\not\in L_1$. Since $\gr(L_1)=\gr(L_2)$, there is a
$c\in L_1\cap B_n$ such that $\gr_n(c)=\gr_n(b)$. By the
minimality of $n$ we conclude that $b-c\in L_1\cap B_{n-1}$.
But this implies $b=(b-c)+c\in L_1$, a contradiction.
\end{proof}

\noindent Suppose $\{L_n:n\ge 0\}$ is a chain of left ideals
of $B$. Then $\{\gr(L_n):n\ge 0\}$ is a chain of left ideals
of $\gr(B,\mcF)$ which becomes stationary provided that
$\gr(B,\mcF)$ is left noetherian. By Lemma~
\ref{comm_lem:gr(L_1)_strictly_contained_in_gr(L_2)}
it follows that $\{L_n:n\ge 0\}$ is stationary. This
proves $B$ to be left noetherian.

\begin{thm}\label{comm_thm:filtered_noetherian_algebras_satisfy_G0}
Let $B$ be a $k$-algebra endowed with a locally finite filtration
$\mcF$ such that the associated graded $k$-algebra $\bar{B}=\gr(B,\mcF)$
is left noetherian. Then $B$ satisfies condition \textbf{(G0)}. In
particular $B$ is a Jacobson algebra and $\End_B(M)$ is algebraic
over~$k$ for every simple $B$-module $M$.
\end{thm}
\begin{proof}
As in \textbf{(G0)} let $\bar{M}$ be a cyclic graded $\bar{B}[X]$-module.
If $\xi\in\bar{M}_0$ is a cyclic vector, then $L=\{p\in\bar{B}[X]:p\mdot\xi=0\}$
is a left ideal of $\bar{B}[X]$ such that $\bar{M}\cong\bar{B}[X]/L$. In
particular it follows that $\bar{M}$ is noetherian. Furthermore the
summands of the grading
\[\bar{M}=\bigoplus_{k=0}^\infty\bar{M}_k=\bigoplus_{k=0}^\infty
\bar{B}_k[X]\mdot\xi\]
are finitely generated $k[X]$-modules. Set $A=k[X]$. Obviously the set
\[\bar{T}=T_A(\bar{M})=\{\;\bar{m}\in\bar{M}:\text{ there exists a
non-zero }g\in A\text{ such that }g\mdot\bar{m}=0\;\}\]
of $A$-torsion elements of $\bar{M}$ is a $\bar{B}[X]$-submodule.
Since $\bar{M}$ is noetherian, we know that
$\bar{T}=\sum_{j=0}^N\bar{B}[X]\mdot\eta_j$ is finitely generated.
As $\eta_j\in\bar{T}$, there exist $0\neq f_j\in A$ such that
$f_j\mdot\eta_j=0$. Setting $f=\prod_{j=0}^N f_j\neq 0$ we find
that $f\mdot\bar{T}=0$. We consider the simple localization
$A_f=S^{-1}A$ of $A$ by $S=\{f^n:n\ge 0\}$. Next we observe
that the $A_f$-torsion of the localization
$\bar{M}_f=A_f\otimes_A\bar{M}$ of $M$ is zero:
\[T_{A_f}(\bar{M}_f)=\{\;\frac{\bar{m}}{s}:\bar{m}\in T_A(\bar{M})
\text{ and }s\in S\;\}=0\,.\]
For $A_f$ is a Pr\"ufer Domain, it follows that $\bar{M}_f$ is a flat
$A_f$-module. Consequently all summands of the decomposition
$\bar{M}_f=\bigoplus_{k=0}^\infty A_f\otimes_A\bar{M}_n$ are flat.
Further the $A_f\otimes_A\bar{M}_n$ are projective because they
are finitely generated over $A_f$. Moreover, they are even
free because $A_f$ is a principal ideal domain. Altogether
we see that $\bar{M}_f$ is a free $A_f$-module. This
proves \textbf{(G0)}. Since $\gr(B,\mcF)$ and hence $B$ are
noetherian, Proposition~\ref{comm_prop:G0_implies_Q0} and
Theorem~\ref{comm_thm:Q0_implies_D} imply that $B$
is a Jacobson algebra satisfying \textbf{(A)}.
\end{proof}

We sustain our search for sufficient conditions for $B$ to be
Jacobson. The aim is to introduce a condition~\textbf{(G1)} on
the level of associated graded algebras and modules which is
stronger than~\textbf{(Q1)}. To this end we suppose that the
$k$-algebra $B'=B[X]$ carries a filtration $\mcF'=\{B'_n:n\ge -1\}$
which need not be induced by a filtration $\mcF$ of $B$ as above.
However, the $k[Y]$-algebra $B'[Y]$ is endowed with the
natural filtration $\mcF''$ induced by $\mcF'$. In particular
$\gr(B'[Y],\mcF'')\cong\gr(B',\mcF')[Y]$.\\\\
Let $M$ be a simple $B'$-module and $\varphi\in\End_{B'}(M)$. As usual we
regard $M$ as a $B'[Y]$-module via $Y\,\mdot\,m=\varphi(m)$ and form the
associated graded $\gr(B'\mcF')[Y]$-module $\bar{M}$. Note that $\bar{M}$
is cyclic over $\gr(B',\mcF')$. As a variant of \textbf{(G0)} we implement

\begin{itemize}
\item[\textbf{(G1)}] Every graded $\gr(B',\mcF')[Y]$-module $\bar{M}$
which is cyclic over $\gr(B',\mcF')$ is generically free over $k[Y]$.
\end{itemize}

\noindent Applying Proposition~\ref{comm_prop:G0_implies_Q0}
to $(B'[Y],\mcF'')$ we obtain

\begin{cor}
Condition \textbf{(G1)} implies \textbf{(Q1)}.
\end{cor}

\begin{rem}
It is an interesting question whether the preceding results can
be generalized to filtered algebras over arbitrary Jacobson rings.
In~\cite{Szczepanski} it is proven that if $A$ is a noetherian
Jacobson ring and $(B,\mcF)$ is an $A$-algebra with a locally
finite filtration such that $\gr(B,\mcF)$ is noetherian, then
$B$ is a Jacobson algebra. The crucial step in the proof is
to show that $A/A\cap P$ is a field for every primitive ideal
$P$ of $B$.
\end{rem}

\vspace{1.5cm}

\subsection*{Appendix A: More on filtered algebras}

\noindent Let $I$ be an ideal of an algebra $B$ with filtration
$\mcF=\{B_n:n\ge -1\}$. Let $\pi:B\to B/I$ denote the canonical
map. Then $\dot{\mcF}=\{\pi(B_n):n\ge -1\}$ is a filtration of $B/I$.
Further the maps $\pi_n:B_n/B_{n-1}\to \pi(B_n)/\pi(B_{n-1})$
define a homomorphism of $\gr(B,\mcF)$ onto $\gr(B/I,\dot{\mcF})$
with kernel $\gr(I)$. From this we deduce that $B/I$ is finitely
generated and almost commutative whenever $B$ is. Further if $B$
is an almost commutative algebra of finite type, so is $B/I$.\\\\
The following observation seems to be appropriate: If $(B,\mcF)$
is a finitely generated, almost commutative $A$-algebra, then
$\gr(B,\mcF)$ is a finitely generated commutative $A$-algebra,
but the converse fails. The notion of an almost commutative
algebra $(B,\mcF)$ of finite type is more general.\\\\
Under the additional assumption that $A$ is a principal ideal
domain, the preceding result can be generalized to the case
of almost commutative $A$-algebras $B$ of finite type. The
next proposition is contained implicitly in
Quillen~\cite{Quillen}. See also Lemme~2.6.4
in~\cite{Dix6}.\\\\
Let $k$ be a field, $A$ a commutative $k$-algebra, and
$(B,\mcF)$ an $A$-algebra with filtration $\mcF=\{B_n:n\ge -1\}$.
Clearly $B'=A\otimes_k B$ becomes an $A$-algebra via
\[(a_1\otimes b_1)(a_2\otimes b_2)=(a_1a_2)\otimes(b_1b_2)
\quad\text{and}\quad a\mdot(a_1\otimes b_1)=(aa_1)\otimes b_1\;.\]
Further $B_n'=A\otimes_k B_n$ defines a filtration $\mcF'$
of the $A$-algebra $B'$. Note that $B'$ can also be regarded
as a filtered $k$-algebra. Since $A\otimes_k-$ is exact and
commutes with direct sums, it follows that
\[B_n'\,/\,B_{n-1}'=(A\otimes_k B_n)\;/\;(A\otimes_k B_{n-1})
\cong A\otimes_k(B_n\,/\,B_{n-1})\]
are isomorphic as $A$-modules, and
\[\gr(B',\mcF')=\bigoplus_{n=0}^\infty\;B_n'\,/\,B_{n-1}'
=A\otimes_k(\;\bigoplus_{n=0}^\infty B_n\,/\,B_{n-1}\;)\cong
A\otimes_k\gr(B,\mcF)\]
as $A$-algebras. From the last equation we deduce
\begin{enumerate}
\item If $(B,\mcF)$ is a (finitely generated) almost commutative
$k$-algebra, then $(B',\mcF')$ is also a (finitely generated)
almost commutative $A$-algebra.
\item If $\gr(B,\mcF)$ is a finitely generated $k$-algebra, then
$\gr(B',\mcF')$ is a finitely generated $A$-algebra. If in
addition $A$ is finitely generated as a $k$-algebra, then
$\gr(B',\mcF')$ is also a finitely generated $k$-algebra.
\end{enumerate}

\subsection*{Appendix B: Localizations}

In this section we collect some results about localizations. In particular
we prove that every localization $S^{-1}A$ of a commutative ring $A$ is a flat
(right) $A$-module.\\\\
Let $A$ be a commutative ring and $S$ a multiplicative subset of $A$ which
means $1\in S$, $0\not\in S$, and $s,t\in S$ implies $st\in S$. In this
section we shall discuss the notion of a localization of $A$ with respect
to $S$. Generalizing the definition of the field of fractions of a
integral domain we consider the following equivalence relation on $S\times A$:
$(s,a)\sim (t,b)$ if and only if there exists $v\in S$ such that
$vta=vsb$. Let $\frac{a}{s}$ denote the equivalence class of $(s,a)$,
and $S^{-1}A$ the set of all equivalence classes. One verifies easily that
\[\frac{a}{s}\;+\;\frac{b}{t}=\frac{ta+sb}{st}\quad\text{and}\quad
\frac{a}{s}\;\mdot\;\frac{b}{t}=\frac{ab}{st}\]
give a well-defined addition and multiplication on $S^{-1}A$. Further
we define a map $\iota:A\to S^{-1}A$, $\iota(a)=\frac{a}{1}$. Then the
localization $(S^{-1}A,\iota)$ of $A$ with respect to $S$ has the
following properties:
\begin{itemize}
\item $S^{-1}A$ is a ring and $\iota:A\to S^{-1}A$ is a ring homomorphism.
\item $\iota(s)$ is invertible for all $s\in S$.
\item If $B$ is a ring and $\varphi:A\to B$ a ring homomorphism such that
$\varphi(s)$ is invertible for all $s\in S$, then there exists a unique
homomorphism $\bar{\varphi}:S^{-1}A\to B$ such that
$\varphi=\bar{\varphi}\circ\iota$.
\item $S^{-1}A$ is generated by $\iota(A)$ and $\iota(S)^{-1}$.
\item $\ker\iota=\{a\in A: as=0\text{ for some }s\in S\}$.
\end{itemize}
Clearly the first three properties determine $(S^{-1}A,\iota)$ up to
isomorphisms. Next we shall discuss the ideal theory of $A$ and $S^{-1}A$.
If $J$ is an ideal of $A$, then
\[\ext(J)=(S^{-1}A)\mdot\iota(J)=\{\;\frac{a}{s}:a\in J\text{ and }s\in S\;\}\]
is an ideal of $S^{-1}A$ called the extension of $J$ in $S^{-1}A$. If $I$
is an ideal of $S^{-1}A$, then
\[\res(I)=\iota^{-1}(I)=\{\;a\in A:\frac{a}{1}\in I\;\}\]
is an ideal of $A$, the restriction of $I$ to $A$. Obviously
\[I=\ext(\res(I))\quad\text{and}\quad J\subset\res(\ext(J)).\]
Note that $I\neq S^{-1}A$ is proper if and only if $\res(I)\cap S=\emptyset$.
If $I$ is prime, so is $\res(I)$. If $J$ is prime and $J\cap S=\emptyset$,
then $\ext(J)$ is also prime. In this case $J=\res(\ext(J))$. Furthermore, if
$J_1\subset J_2$, then $\ext(J_1)\subset \ext(J_2)$, and if $I_1\subset I_2$,
then $\res(I_1)\subset\res(I_2)$.

\begin{defn}
Let $A$ be a commutative ring. Let $\Spec(A)$ denote the set of all prime ideals
of $A$. If $J$ is an ideal of $A$, then $h(J)=\{P\in\Spec(A):J\subset P\}$ is
called the hull of $J$. Conversely, if $X\subset\Spec(A)$, then the ideal
$k(X)=\bigcap\{P:P\in X\}$ is called the kernel of $X$. We say that a subset
$X$ of $\Spec(A)$ is closed if and only if $X=h(J)$ for some ideal $J$ of
$A$. The space $\Spec(A)$ endowed with this so-called hull-kernel topology
is the spectrum of $A$.
\end{defn}

\noindent The preceding observations show that the spectrum $\Spec(S^{-1}A)$
of the localization of $A$ with respect to $S$ can be identified with the
subset $\{P\in\Spec(A):P\cap S=\emptyset\}$ of the spectrum of $A$ by means
of the homeomorphisms $\ext$ and $\res$.

\begin{lem}\label{comm_lem:localizations_of_principal_ideal_domains}
If $A$ is a principal ideal domain, so is $S^{-1}A$. 
\end{lem}
\begin{proof}
Obviously $S^{-1}A$ is an integral domain. Let $I$ be an ideal of $S^{-1}A$. Since
$A$ is a principal ideal ring, there exists some $b\in A$ such that
$\res(I)=A\mdot b$. This implies
\[I=\ext(\res(I))=(S^{-1}A)\;\mdot\;\frac{b}{1}.\]
\end{proof}

Now let $P$ be a left $A$-module. Then the functors $\Hom_A(-,P)$ and $\Hom_A(P,-)$
are left exact. One can prove that $P\otimes_A -$ is also right exact.

\begin{defn}
Let $A$ be a ring and $P$ a right $A$-module. We say that $P$ is a flat
$A$-module if the functor $P\otimes_A-$ is exact.
\end{defn}

Since $P\otimes_A-$ is already known to be right exact, it follows
that $P$ is $A$-flat if and only if the following condition is satisfied: If
$\varphi:L\to M$ is an injective homomorphism of left $A$-modules, then
$1\otimes\varphi:P\otimes_A L\to P\otimes_A M$ is injective.\\\\
To prove that localizations $S^{-1}A$ are $A$-flat, i.e., that the functor
$(S^{-1}A)\otimes_A-$ is exact, we introduce localizations of $A$-modules.
Let $A$ be a commutative ring, $S$ a multiplicative subset of $A$, and $M$ a
left $A$-module. In analogy to the definition of $S^{-1}A$ we consider the
following equivalence relation on $S\times M$: $(s,m)\sim (t,n)$ if and only
if there exists a $v\in S$ such that $vt\mdot m=vs\mdot n$. Let $\frac{m}{s}$
denote the equivalence class of $(s,m)$, and $S^{-1}M$ the set of all such
equivalence classes. It is easy to see that
\[\frac{m}{s}\;+\;\frac{n}{t}\;=\;\frac{tm+sn}{st}\quad\text{and}\quad
\frac{a}{r}\;\mdot\;\frac{m}{s}\;=\;\frac{a\mdot m}{rs}\]
gives a well-defined $S^{-1}A$-module structure on $S^{-1}M$.

\begin{lem}\label{comm_lem:localization_of_modules}
Let $A$ be a commutative ring and $S$ a multiplicative subset of $A$. If $M$ is
a left $A$-module, then $(S^{-1}A)\otimes_A M$ and $S^{-1}M$ are isomorphic as
left $S^{-1}A$-modules.
\end{lem}
\begin{proof}
By the universal property of the balanced tensor product we obtain a map
\[\Phi:(S^{-1}A)\otimes_A M\to S^{-1}M\;,\;\Phi(\;\frac{a}{r}\otimes m\;)=
\frac{a\mdot m}{r}\;.\]
One checks that $\Phi$ is a homomorphism of left $S^{-1}A$-modules. On the other
hand we have a map
\[\Psi:S^{-1}M\to(S^{-1}A)\otimes_A M\;,\;\Psi(\;\frac{m}{s}\;)=
\frac{1}{s}\otimes m\;.\]
We check that $\Psi$ is well-defined: Let $m,n\in M$ and $s,t\in S$ such that
$\frac{m}{s}=\frac{n}{t}$. Hence there exists a $v\in S$ such that
$vt\mdot m=vs\mdot n$. This implies
\[\Psi(\frac{m}{s})=\frac{1}{s}\otimes m=\frac{1}{vts}\otimes( vt\mdot m)
=\frac{1}{vts}\otimes(vs\mdot n)=\frac{1}{t}\otimes n=\Psi(\frac{n}{t})\;.\]
Obviously $\Psi$ is also a homomorphism of $S^{-1}A$-modules. Further
$\Psi\circ\Phi=\Id$ and $\Phi\circ\Psi=\Id$.
\end{proof}

\begin{prop}\label{comm_prop:localizations_are_flat}
Let $A$ be a commutative ring and $S$ a multiplicative subset of $A$. Then
$S^{-1}A$ is a flat $A$-module.
\end{prop}
\begin{proof}
Let $\varphi:L\to M$ be an injective homomorphism of left $A$-modules.
By the preceding considerations it suffices to prove that
$1\otimes\varphi:(S^{-1}A)\otimes_AL\to(S^{-1}A)\otimes_A M$ is injective.
By means of the isomorphisms $\Phi$ and $\Psi$ given in
Lemma~\ref{comm_lem:localization_of_modules} we see that $1\otimes\varphi$
corresponds to the homomorphism
\[\bar{\varphi}:S^{-1}L\to S^{-1}M\;,\;\bar{\varphi}(\;\frac{l}{s}\;)=
\frac{\varphi(l)}{s}\;.\]
It suffices to check that $\bar{\varphi}$ is injective: Let $l\in L$ and $s\in S$
such that $0=\bar{\varphi}(\frac{l}{s})=\frac{\varphi(l)}{s}$. Then there exists
$v\in S$ such that $0=v\mdot\varphi(l)=\varphi(v\mdot l)$. Hence $0=v\mdot l$ because
$\varphi$ is injective. This proves $\frac{l}{s}=\frac{0}{1}$.
\end{proof}

\subsection*{Appendix C: Semiprimitive and semiprime ideals}

Here we collect some properties of (semi-)primitive and (semi-)prime ideals in non-commutative
rings.

\begin{defn}
Let $B$ be a ring and $\Lambda$ a subset of $B$. We say that $\Lambda$ is $m$-closed if,
for any $a,b\in\Lambda$, there exists some $x\in B$ such that $axb\in\Lambda$. Further
$\Lambda$ is called $p$-closed if, for any $a\in\Lambda$, there exists some $x\in B$ such
that $axa\in\Lambda$.
\end{defn}

\noindent Prime ideals are characterized easily as follows.

\begin{prop}\label{comm_prop:characterization_of_prime_ideals}
Let $B$ be a ring and $I$ an ideal of $B$ with $I\neq B$. Then there are equivalent:
\begin{enumerateroman}
\item $I$ is prime.
\item $a_1Ba_2\subset I$ implies $a_1\in I$ or $a_2\in I$.
\item $B\setminus I$ is an $m$-closed subset of $B$.
\end{enumerateroman}
\end{prop}
\begin{proof}
First we prove \textit{(i)}$\Rightarrow$\textit{(ii)}. Suppose that $I$ is prime. Let
$a_1,a_2$ be in $B$ such that $a_1Ba_2\subset I$. Then $(Ba_1B)(Ba_2B)\subset I$. Since
$I$ is prime, it follows $Ba_1B\subset I$ or $Ba_2B\subset I$, and hence $a_1\in I$ or
$a_2\in I$ because $B$ is unital. This proves \textit{(ii)}. Obviously \textit{(iii)} is
the contraposition of \textit{(ii)}. Thus \textit{(ii)}$\Leftrightarrow$\textit{(iii)}.
It remains to prove \textit{(ii)}$\Rightarrow$\textit{(i)}. Let $J_1$ and $J_2$ be
ideals of $B$ such that $J_1J_2\subset I$. Suppose that $J_2\not\subset I$ and choose
$b\in J_2\setminus I$. For every $a\in J_1$ we have $aBb\subset J_1J_2\subset I$ and
thus $a\in I$ by~\textit{(ii)}. This proves $J_1\subset I$. Hence $I$ is prime.
\end{proof}

\noindent Before we state a similar characterization for semiprime ideals, we
prove the following auxiliary result.

\begin{lem}\label{comm_lem:pclosed_subsets_contain_mclosed_subsets}
Let $B$ be a ring. If $\Lambda$ is a $p$-closed subset of $B$ and $x\in\Lambda$, then
there exists a countable $m$-closed subset $\Lambda_0$ of $B$ with $\Lambda_0\subset\Lambda$
and $x\in\Lambda_0$.
\end{lem}
\begin{proof}
By induction we define a sequence $\{x_n:n\ge 0\}$ of elements of $\Lambda$ as follows: Set
$x_0=x$. If $x_0,\ldots,x_n$ are defined, then there is a $y\in B$ such that $x_nyx_n\in\Lambda$
because $\Lambda$ is $p$-closed. We set $x_{n+1}=x_nyx_n$. Finally we define
$\Lambda_0=\{x_n:n\ge 0\}$. Now we must prove that $x_mBx_n\cap\Lambda_0\neq\emptyset$ for
all $m,n\ge 0$: Suppose that $m\le n$. Then $x_{n+1}\in x_nBx_n\subset x_mBx_n$ and
$x_{n+1}\in\Lambda_0$. The case $m\ge n$ is similar.
\end{proof}

\begin{prop}\label{comm_prop:characterization_of_semiprime_ideals}
Let $B$ be a ring and $I\neq B$ an ideal of $B$. Then there are equivalent:
\begin{enumerateroman}
\item $I$ is semiprime.
\item If $J$ is an ideal of $B$ such that $J^2\subset I$, then $J\subset I$.
\item $aBa\subset I$ implies $a\in I$.
\item $B\setminus I$ is a $p$-closed subset of $B$.
\end{enumerateroman}
\end{prop}
\begin{proof}
First we prove \textit{(i)}$\Rightarrow$\textit{(ii)}. Suppose that $I$ is semiprime and
let $J$ be an ideal of~$B$ such that $J^2\subset I$. Let $P$ be an arbitrary prime ideal
of $B$ such that $I\subset P$. From $J^2\subset P$ it follows $J\subset P$. Intersecting
all these $P$ we conclude $J\subset\sqrt{I}=I$ which proves~\textit{(ii)}. Next we
verify \textit{(ii)}$\Rightarrow$\textit{(iii)}. Let $a\in B$ such that $aBa\subset I$.
Then it follows $(BaB)(BaB)\subset I$ and hence $BaB\subset I$ by~\textit{(ii)}.
Thus $a\in I$ because $B$ is unital. This proves~\textit{(iii)}. Obviously \textit{(iv)}
is the contrapostion of~\textit{(iii)} so that \textit{(iii)}$\Leftrightarrow$\textit{(iv)}.
Finally we establish \textit{(iv)}$\Rightarrow$\textit{(i)}. Suppose that $B\setminus I$
is $p$-closed. We must show that $\sqrt{I}\subset I$. Let $a\in B\setminus I$. From
Lemma~\ref{comm_lem:pclosed_subsets_contain_mclosed_subsets} we deduce that there
exists an $m$-closed subset $\Lambda_0$ of $B\setminus I$ such that $a\in\Lambda_0$.
By Proposition~\ref{comm_prop:characterization_of_prime_ideals} we know that
$P=B\setminus\Lambda_0$ is prime. Since $a\not\in P$, we conclude $a\not\in\sqrt{I}$.
\end{proof}

%{\small
%\bibliographystyle{plain}
%\bibliography{literature}
%}

\end{document}